\documentclass{amsart}
\usepackage{amsthm}
\newtheorem{theorem}{Theorem}[section]
\newtheorem{lemma}[theorem]{Lemma}

\theoremstyle{definition}
\newtheorem{definition}[theorem]{Definition}

\theoremstyle{remark}
\newtheorem{remark}[theorem]{Remark}

\usepackage{amssymb}
\usepackage{empheq}
\usepackage{stmaryrd}
\usepackage{enumerate}
\usepackage[shortlabels]{enumitem}
\usepackage{calc}
\usepackage{url}

\newcommand{\norm}[1]{\left\lVert#1\right\rVert}
\usepackage{mathtools}
\DeclarePairedDelimiter{\ceil}{\lceil}{\rceil}

\newcommand{\remin}{\mathop{-\!\!\!\!\!\hspace*{1mm}\raisebox{0.5mm}{$
\cdot$}}\nolimits}

\usepackage[left=2.0cm,%
                right=2.0cm,%
                top=2.5cm,%
                bottom=3.5cm,%
                headheight=12pt,%
                a4paper]{geometry}%

\begin{document}

\title[Quantitative results on differences of monotone operators]{Quantitative results on algorithms for zeros of differences of monotone operators in Hilbert space}

\author[Nicholas Pischke]{Nicholas Pischke}
\date{\today}
\maketitle
\vspace*{-5mm}
\begin{center}
{\scriptsize Department of Mathematics, Technische Universit\"at Darmstadt,\\
Schlossgartenstra\ss{}e 7, 64289 Darmstadt, Germany, \ \\ 
E-mail: pischkenicholas@gmail.com}
\end{center}

\maketitle
\begin{abstract}
We provide quantitative information in the form of a rate of metastability in the sense of T. Tao and (under a metric regularity assumption) a rate of convergence for an algorithm approximating zeros of differences of maximally monotone operators due to A. Moudafi by using techniques from `proof mining', a subdiscipline of mathematical logic. For the rate of convergence, we provide an abstract and general result on the construction of rates of convergence for quasi-Fej\'er monotone sequences with metric regularity assumptions, generalizing previous results for Fej\'er monotone sequences due to U. Kohlenbach, G. L\'opez-Acedo and A. Nicolae.
\end{abstract}

\section{Introduction}
In \cite{Mou2015}, Moudafi introduced an algorithm for approximating critical points of differences of two maximally monotone operators by generalizing the method of proximity operators in DC programming (see below). Specifically, given two maximally monotone operators $T,S$ on a real Hilbert space $X$, one wants to find points
\[
x^*\in\Gamma:=\{x^*\in X\mid T(x^*)\cap S(x^*)\neq\emptyset\}.
\]

If $T$ and $S$ are both subdifferentials of convex functions $f,g$, respectively, this covers the prominent case of \emph{DC programming}, i.e. mathematical programming for \emph{differences of convex functions} (see e.g. \cite{HT1999,TD2018}) and the solution will be a critical point of $f-g$.

The algorithm given in \cite{Mou2015} is motivated by noting that $\mathrm{gra}{T_\lambda}\to\mathrm{gra}T$ for $\lambda\to 0$ where $T_{\lambda}(x)=\frac{x-J^T_{\lambda}x}{\lambda}$ is the Yosida approximate of $T$ and $J^T_{\lambda}=(Id+\lambda T)^{-1}$ is the resolvent of $\lambda T$. This leads to regularizing the above problem to finding points $x_\lambda$ with
\[
T_\lambda(x_\lambda)\in S(x_\lambda)
\]
for $\lambda\to 0$. This inclusion can be equivalently phrased as a fixed point problem
\[
x_\lambda=J^S_\mu(x_\lambda+\mu T_\lambda x_\lambda)
\]
with parameter $\mu>0$. This then leads to the iteration scheme
\[
x_{n+1}=J^S_{\mu_n}(x_n+\mu_n T_{\lambda_n}x_n)
\]
given initial data $x_0$ and parameters $\mu_n,\lambda_n>0$ which, under suitable assumptions on the parameters, can be shown to be convergent:
\begin{theorem}[\cite{Mou2015}]\label{thm:MoudafiConv}
Let $T,S$ be two maximally monotone operators on a finite dimensional Hilbert space $X$ such that $\Gamma\neq\emptyset$, $T$ is bounded on bounded sets and $\mathrm{dom}S\subseteq\mathrm{dom}T$ as well as
\begin{enumerate}
\item $\lim_{n\to\infty}\lambda_n=0$,
\item $\sum_{n=0}^\infty\frac{\mu_n}{\lambda_n}<\infty$,
\item $\lim_{n\to\infty}\norm{x_n-x_{n+1}}/\mu_n=0$,
\end{enumerate}
for $\lambda_n,\mu_n>0$. Then $(x_n)$ converges to a point $x^*\in\Gamma$.
\end{theorem}
Moudafi's results does not give any quantitative information on the convergence. However, by inspection of the proof given in \cite{Mou2015}, it becomes apparent that it relies on a standard argument via establishing (quasi-)Fej\'er monotonicity (see \cite{Com2001}) and then inferring convergence from that.\footnote{Although Fej\'er monotonicity is not mentioned explicitly in \cite{Mou2015}.} In terms of quantitative results, this opens the door for applying the recent results of Kohlenbach, Leu\c{s}tean and Nicolae \cite{KLN2018} as well as of Kohlenbach, L\'opez-Acedo and Nicolae \cite{KLAN2019} on the finitary content of convergence of (quasi-)Fej\'er monotone sequences, which has been successfully applied in many other contexts of nonlinear analysis, in particular for the asymptotic regularity of compositions of two mappings \cite{KLAN2017}, the proximal point algorithm in uniformly convex Banach spaces \cite{Koh2021} and subgradient methods for equilibrium problems \cite{PK2021}.\\

The results of the papers \cite{KLN2018,KLAN2019} were obtained via the general methodological approach of `proof mining', a subdiscipline of mathematical logic which aims at the extraction of quantitative information from prima facie nonconstructive proofs by logical transformations (see \cite{Koh2008} for a book treatment and \cite{Koh2019} for a recent survey). This approach has also been instrumental for obtaining the present results.\\

In terms of quantitative information, even for computable Fej\'er monotone sequences of real numbers, in general, there exists no computable rate of convergence which follows from fundamental results in recursion theory (see \cite{KLN2018,Neu2015}).\footnote{These results extend the phenomena of `arbitrary slow convergence' from optimization.} 

However, in very general situations, one can extract effective rates of so-called metastability from non-effective proofs of convergence which are, moreover, highly uniform. This notion of metastability originates from a (noneffectively) equivalent reformulation of the Cauchy property in some metric space $(X,d)$
\[
\forall k\in\mathbb{N}\exists n\in\mathbb{N}\forall i,j\geq n\left(d(x_i,x_j)< \frac{1}{k+1}\right)
\]
into 
\[
\forall k\in\mathbb{N}\forall g\in\mathbb{N}^\mathbb{N}\exists n\in\mathbb{N}\forall i,j\in [n;n+g(n)]\left(d(x_i,x_j)< \frac{1}{k+1}\right)
\]
where $[n;n+m]:=\{n+i\mid i\in\mathbb{N}\land 0\leq i\leq m\}$. From a logical perspective, this reformulation can be recognized as the so-called Herbrand normal form of (a slightly tweaked version of) the above Cauchy property which is in particular of the general form $\forall\exists$ (considering the leading two universal quantifiers as one and disregarding the last universal quantifier as it is bounded) and for statements of the above form, the logical metatheorems of proof mining guarantee the extractability of a rate of metastability, that is a (highly uniform and effective) bound on `$\exists n\in\mathbb{N}$' in the above reformulation (see \cite{Koh2008}). This notion has also been recognized as an important finitary version of the Cauchy property from a non-logical perspective by Tao (see e.g. \cite{Tao2008b,Tao2008a}) who actually coined the term metastability.\\ 

In this paper, we provide a further case study to illustrate how the abstract approach from \cite{KLN2018,KLAN2019} can be used in a particular scenario by obtaining an explicit quantitative version of Theorem \ref{thm:MoudafiConv} in the form of a fully effective and highly uniform rate of metastability. Moreover, in the latter parts of the paper, we even give a rate of convergence, modulo an additional metric regularity assumption in the sense of \cite{KLAN2019}. For this, we extend the main result from \cite{KLAN2019} to the case of quasi-Fej\'er monotone sequences. The quantitative analysis in particular relies on a mild form of uniform continuity of a set-valued operator $T$ similar to that introduced in \cite{KP2020} and we expect that this quantitative notion and its use, together with the whole approach, detailed here will provide a guideline for future analyses of other convergence results from monotone operator theory, especially works dealing with differences of monotone operators like \cite{AO1999,Mou2008,NNHS2009,RAC2020}. In fact, the analysis presented here immediately generalizes to the extensions for Moudafi's result considered in \cite{RAC2020} where one similarly obtains a simple rate of metastability and even a rate of convergence under a metric regularity assumption but we omit any details regarding this.

\section{Quasi-Fej\'er monotonicity, uniform continuity and rates of metastability}
As mentioned in the introduction, the proof of Theorem \ref{thm:MoudafiConv} given in \cite{Mou2015}, and with that the following quantitative analysis, relies on the notion of quasi-Fej\'er monotonicity (see again \cite{Com2001}) which we want to briefly recall. For this, we actually rely on the following generalized version introduced in \cite{KLN2018}. 
\begin{definition}[\cite{KLN2018}]
Let $G:\mathbb{R}_+\to\mathbb{R}_+$ and $H:\mathbb{R}_+\to\mathbb{R}_+$ be functions where
\[
a_n\to 0\text{ implies }G(a_n)\to 0\text{ and }H(a_n)\to 0\text{ implies }a_n\to 0
\]
for any sequence $(a_n)$ from $\mathbb{R}_+$ and let $(X,d)$ be a metric space, $F\subseteq X$ be nonempty and $(x_n)$ be a sequence in $X$. $(x_n)$ is called \emph{quasi-$(G,H)$-Fej\'er monotone with respect to $F$}, if 
\[
\forall n,m\in\mathbb{N}\forall p\in F\left(H(d(x_{n+m},p))\leq G(d(x_n,p))+\sum_{i=n}^{n+m-1}\varepsilon_i\right),
\]
where $(\varepsilon_i)\subseteq\mathbb{R}_+$ is such that $\sum_i\varepsilon_i<\infty$.
\end{definition}
For formulating the quantitative results, we pass from quasi-Fej\'er monotonicity to \emph{uniform} quasi-Fej\'er monotonicity with an accompanying modulus in the sense of \cite{KLN2018}. To do this, we assume a corresponding stratification of $F$ by sets $AF_k$ s.t.
\[
AF_k\supseteq AF_{k+1}\text{ and }F=\bigcap_{k\in\mathbb{N}}AF_k.
\]
Intuitively, the sets $AF_k$ are meant to represent the set of $k$-good approximations of the set $F$ and can take many forms in an actual application.
\begin{definition}[\cite{KLN2018}]
Let $G,H$ be as before. Then $(x_n)$ is called \emph{uniformly quasi-$(G,H)$-Fej\'er monotone with respect to $F$} (\emph{and $(AF_k)$}) if for all $r,n,m\in\mathbb{N}$:
\[
\exists k\in\mathbb{N}\forall p\in AF_k\forall l\leq m\left(H(d(x_{n+l},p))<G(d(x_n,p))+\sum_{i=n}^{n+l-1}\varepsilon_i+\frac{1}{r+1}\right).
\]
Any function $\chi(n,m,r)$ producing an upper bound on such a $k\in\mathbb{N}$ is called a \emph{modulus of uniform quasi-$(G,H)$-Fej\'er monotonicity for $(x_n)$}.
\end{definition}
As a second ingredient, we need quantitative information on how the sequence $(x_n)$ approaches the set $F$ w.r.t. the stratification $AF_k$, in the sense of the following definition:
\begin{definition}[\cite{KLN2018}]
$(x_n)$ has the \emph{$\liminf$-property w.r.t. $F$ (and $(AF_k)$)} if $\forall k,n\in\mathbb{N}\exists N\geq n\left(x_N\in AF_k\right)$. A bound $\Phi(k,n)$ on $N$, which is monotone in $k$ and $n$, is called a \emph{$\liminf$-bound for $(x_n)$}.
\end{definition}
Using the proof mining macro established in \cite{KLN2018}, one can then combine these moduli, together with some further (minor) quantitative assumptions on the surrounding data, to a rate of metastability of the sequence. As discussed in \cite{KLN2018}, these moduli are guaranteed to exist (in very general situation) by logical metatheorems and can often be obtained by a separate application of corresponding logical bound extraction results of proof mining and it is the extraction of these moduli from the proof given in \cite{Mou2015} which we detail (without any reference to logic) in the following.\\

For that, we will in particular rely on a certain notion of uniform continuity for a set-valued operator $T$ which generalizes the usual notion of uniform continuity (as, e.g., stipulated in \cite{MN2001}) for a set-valued operator $T$
\[
\forall\varepsilon>0\exists\delta>0\forall x,y\in\mathrm{dom}T\left(\norm{x-y}\leq\delta\rightarrow H(Tx,Ty)\leq\varepsilon\right)
\]
where $H$ is the Hausdorff-metric.\footnote{We have used $H$ for another object before but the context will make it clear whether the Hausdorff-metric is meant.} Motivated by logical considerations (see Remark \ref{rem:logic}), \cite{KP2020} introduced an `approximate version' of the Hausdorff metric in the form of a \emph{Hausdorff-like predicate} $H^*$ defined via
\[
H^*[P,Q,\varepsilon]:=\forall p\in P\exists q\in Q\left(\norm{p-q}\leq\varepsilon\right).
\]
One can then stipulate uniform continuity w.r.t. that predicate $H^*$ by requiring
\[
\forall\varepsilon>0\exists\delta>0\forall x,y\in X\left(\norm{x-y}\leq\delta\rightarrow H^*[Tx,Ty,\varepsilon]\right).
\]
We say that $\varpi:\mathbb{N}\to\mathbb{N}$ is a \emph{modulus of uniform continuity for $T$ w.r.t. $H^*$} if
\[
\forall k\in\mathbb{N}\forall x,y\in X\left(\norm{x-y}\leq\frac{1}{\varpi(k)+1}\rightarrow H^*\left[Tx,Ty,\frac{1}{k+1}\right]\right)
\]
Our analysis will in the following rely on such a modulus of uniform continuity for $T$ w.r.t. $H^*$ (see again Remark \ref{rem:logic}). For convenience, we will assume that $\varpi$ is monotone increasing.\\

Now, assume that $\mathrm{dom}S\subseteq\mathrm{dom} T$ and let $L\geq\mathrm{diam}\{x_n\mid n\in\mathbb{N}\}$ for a concrete sequence $(x_n)$ of the algorithm. Throughout, we will actually work over the compact space $X_0=\overline{B}(x_0;L)\cap\overline{\mathrm{dom}S}$ and all sets and moduli are to be understood as being relativized to this set. We first define appropriate instantiations $\Gamma_k$ of the abstract approximations $AF_k$ discussed before in the context of Moudafi's algorithm by setting
\begin{align*}
&\Gamma_k:=\bigg\{x^*\in X_0\mid \exists y^*\bigg(\vert\norm{y^*}-\norm{T^\circ x^*}\vert\leq\frac{1}{k+1}\land H^*\left[y^*,Tx^*,\frac{1}{k+1}\right]\\
&\qquad\qquad\qquad\qquad\qquad\qquad\land \forall i\leq k\left(\norm{x^*-J^S_{\mu_i}(x^*+\mu_iy^*)}\leq\frac{1}{k+1}\right)\bigg)\bigg\}
\end{align*}
where $H^*$ is the previously discussed Hausdorff-like predicate (where we write $y^*$ for the singleton $\{y^*\}$) and $T^\circ x=P_{Tx}0$ is the element of minimal norm in $Tx$ (see, e.g., \cite{BC2017}). We write $\Gamma_k(x^*)$ for the set of all such $y^*$ realizing the existential quantifier in the above definition with parameter $x^*$.\\

The motivation for this particular stratification of the set of solutions $\Gamma$ very much follows the reasoning given by Moudafi in \cite{Mou2015} for the algorithm as discussed in the introduction: $T(x^*)\cap S(x^*)$ is regularized to $T_\lambda(x_\lambda)\in Sx_\lambda$ for $\lambda\to 0$, which is equivalent to solving the fixed-point equation $x_\lambda=J^S_\mu(x_\lambda+\mu T_\lambda x_\lambda)$. To formalize what it means for $x^*$ to be a $k$-good approximation of a solution in $\Gamma$, we replace $T_\lambda(x_\lambda)$, itself an approximation for a point in the intersection, by a generic point $y^*$ which is supposed to be a $k$-good approximation of $T_\lambda(x^*)$ for suitable $\lambda$, formalized via being close to $Tx^*$ in the sense of the Hausdorff-like predicate. In the same way, the fixed point condition $x_\lambda=J^S_\mu(x_\lambda+\mu T_\lambda x_\lambda)$ is relativized to an approximate fixed point condition where the parameter $\mu$ is replaced by the sequence $\mu_i$ from the algorithm. In that way, the fixed-point perspective of this algorithm is essential for the following quantitative analysis as it lends itself to useful approximate versions.\\

We begin by showing that the $\Gamma_k$ are appropriate approximate versions of $\Gamma$. For that, we actually show two things. First, we have $\bigcap_k\Gamma_k\subseteq\Gamma$, i.e. an arbitrarily good approximation in the sense of the $\Gamma_k$ is actually a solution. Second, the $\Gamma_k$ are good approximate sets in the sense that $\bigcap_k\Gamma_k$ is uniformly closed w.r.t. the $\Gamma_k$ in the following sense:
\begin{definition}[\cite{KLN2018}]
$F$ is called \emph{uniformly closed} (w.r.t. $AF_k$) with moduli $\delta_F,\omega_F:\mathbb{N}\to\mathbb{N}$ if
\[
\forall k\in\mathbb{N}\forall p,q\in X\left( q\in AF_{\delta_F(k)}\land d(p,q)\leq\frac{1}{\omega_F(k)+1}\rightarrow p\in AF_k\right).
\]
\end{definition}
Before we can state the corresponding result on uniform closedness, we need the following lemma which derives a modulus of continuity for $T^\circ$ from the modulus of uniform continuity $\varpi$ for $T$.
\begin{lemma}\label{lem:Tcircmodunifcont}
\begin{enumerate}
\item For all $k\in\mathbb{N}$ and $x,z\in X$, if $z\in Tx$ and $\langle T^\circ x-z,-z\rangle\leq\frac{1}{(k+1)^2}$, then $\norm{T^\circ x-z}\leq\frac{1}{k+1}$.
\item For all $k\in\mathbb{N}$ and $x,z\in X$, if $z\in Tx$ and $\norm{z}^2-\norm{T^\circ x}^2\leq\frac{1}{(k+1)^2}$, then $\norm{T^\circ x-z}\leq\frac{1}{k+1}$.
\item Let $x,x'\in\mathrm{dom} T$ and $B\geq \norm{T^\circ x'}$ with $B\in\mathbb{N}^*$ and let $\varpi$ be a modulus of uniform continuity for $T$. Then, $\varpi'(k)=\varpi(Bk^2+2Bk+B-1)$ satisfies
\[
\norm{x-x'}\leq\frac{1}{\varpi'(k)+1}\rightarrow\norm{T^\circ x-T^\circ x'}\leq\frac{1}{k+1}\tag{$\dagger$}
\]
for all $x$.
\end{enumerate}
\end{lemma}
\begin{proof}
Note that as $T^\circ x=P_{Tx}(0)$, Theorem 3.16 of \cite{BC2017} yields $\langle y-T^\circ x,-T^\circ x\rangle\leq 0$ for all $x,y$ with $y\in Tx$.
\begin{enumerate}
\item Let $z\in Tx$ and $\langle T^\circ x-z,-z\rangle\leq \frac{1}{(k+1)^2}$. By ($\dagger$), we have $\langle z-T^\circ x,-T^\circ x\rangle\leq 0$. Thus
\begin{align*}
\frac{1}{(k+1)^2}\geq\langle T^\circ x-z,-z\rangle+\langle z-T^\circ x,-T^\circ x\rangle =\norm{T^\circ x-z}^2
\end{align*}
which implies $\norm{T^\circ x-z}\leq\frac{1}{k+1}$.
\item We skip the proof as it is an easy consequence of (1).
\item Let $\norm{x-x'}\leq\frac{1}{\varpi'(k)+1}$. By definition of $\varpi'$, we have that
\[
\exists y'\in Tx'\left(\norm{T^\circ x-y'}\leq\frac{1}{Bk^2+2Bk+B}\right).
\]
Now, using $(\dagger)$ and the Cauchy-Schwarz inequality we obtain
\begin{align*}
\langle T^\circ x-T^\circ x',-T^\circ x'\rangle&\leq\langle T^\circ x-y',-T^\circ x'\rangle+\langle y'-T^\circ x',-T^\circ x'\rangle\\
&=\langle T^\circ x-y',-T^\circ x'\rangle\\
&\leq\norm{T^\circ x-y'}\norm{T^\circ x'}\\
&\leq\frac{1}{(k+1)^2}.
\end{align*}
Item (1) now yields $\norm{T^\circ x-T^\circ x'}\leq\frac{1}{k+1}$.
\end{enumerate}
\end{proof}
\begin{lemma}\label{lem:intersectionGamma}
Let $\varpi$ be a modulus of uniform continuity for $T$ w.r.t. $H^*$ and let $M\in\mathbb{N}^*$ satisfy $M\geq\norm{T^\circ x^*}$ for all $x^*\in X_0$. Then $\bigcap_k\Gamma_k$ is uniformly closed w.r.t. $\Gamma_k$ with moduli $\delta(k)=2k+1$ and $\omega(k)=\max\{4k+3,\varpi(4M(k+1)^2-1)\}$ and, further,
\[
\Gamma\supseteq\bigcap_{k}\Gamma_k.
\]
\end{lemma}
\begin{proof}
Suppose $q\in\Gamma_{\delta(k)}$ and let $a\in\Gamma_k(q)$, i.e. $\vert\norm{a}-\norm{T^\circ q}\vert\leq\frac{1}{\delta(k)+1}$, $H^*[a,Tq,1/(\delta(k)+1)]$, and $\norm{q-J^S_{\mu_i}(q+\mu_i a)}\leq\frac{1}{\delta(k)+1}$ for all $i\leq\delta(k)$.

Let $p\in X_0$ be such that $\norm{p-q}\leq\frac{1}{\omega(k)+1}$. Then in particular $\omega(k)\leq\varpi(2k+1)$ as $\varpi$ is monotone increasing and therefore
\[
\norm{p-q}\leq\frac{1}{\omega(k)+1}\leq\frac{1}{\varpi(2k+1)+1}
\]
and therefore
\[
H^*\left[a,Tp,\frac{1}{\delta(k)+1}+\frac{1}{2(k+1)}\right],\text{ and so }H^*\left[a,Tp,\frac{1}{k+1}\right].
\]
Further, as $\delta(k)\geq k$, we have
\begin{align*}
\norm{p-J^S_{\mu_i}(p+\mu_i a)}&\leq\norm{p-q}+\norm{q-J^S_{\mu_i}(q+\mu_i a)}+\norm{J^S_{\mu_i}(q+\mu_i a)-J^S_{\mu_i}(p+\mu_i a)}\\
&\leq 2\norm{p-q}+\norm{q-J^S_{\mu_i}(q+\mu_i a)}\\
&\leq \frac{2}{\omega(k)+1}+\frac{1}{\delta(k)+1}\\
&\leq \frac{2}{4(k+1)}+\frac{1}{2(k+1)}\\
&=\frac{1}{k+1}
\end{align*}
for all $i\leq k$ using nonexpansivity of $J^S_{\mu_i}$.

Lastly, we have
\begin{align*}
\vert\norm{a}-\norm{T^\circ p}\vert&\leq\vert\norm{a}-\norm{T^\circ q}\vert+\vert\norm{T^\circ p}-\norm{T^\circ q}\vert\\
&\leq\frac{1}{\delta(k)+1}+\norm{T^\circ p-T^\circ q}\\
&\leq\frac{1}{2(k+1)}+\frac{1}{2(k+1)}\\
&=\frac{1}{k+1}
\end{align*}
using the modulus of uniform continuity for $T^\circ$ derived from $\varpi$ in Proposition \ref{lem:Tcircmodunifcont} (where one has to note, in particular, the definition of $M$). Combined, we have $a\in\Gamma_k(p)$ and thus $p\in\Gamma_k$ and $\Gamma$ is therefore uniformly closed w.r.t. $\Gamma_k$ (over $X_0$).\\

For the second claim, let $x^*\in\bigcap_k\Gamma_k$, i.e. for any $k$ there are $y^*_k\in\Gamma_k(x^*)$ such that $\vert\norm{y^*_k}-\norm{T^\circ x^*}\vert\leq\frac{1}{k+1}$, $H^*[y^*_k,Tx^*,1/(k+1)]$ and $\norm{x^*-J^S_{\mu_i}(x^*+\mu_i y^*_k)}\leq\frac{1}{k+1}$ for all $i\leq k$.

Now, item (1) yields $\norm{y^*_k}\to\norm{T^\circ x^*}$ for $k\to\infty$. Item (2) yields
\[
\forall k\exists z^*_k\in Tx^*\left(\norm{y^*_k-z^*_k}\leq\frac{1}{k+1}\right)
\]
and thus in particular $\vert\norm{z^*_k}-\norm{y^*_k}\vert\leq\frac{1}{k+1}$. Thus $\norm{z^*_k}\to\norm{T^\circ x^*}$ for $k\to\infty$. Uniqueness of $T^\circ x^*$ as the element of minimal norm, Lemma \ref{lem:Tcircmodunifcont}, (2), now yields $z^*_k\to T^\circ x^*$ for $k\to\infty$ (actually in a quantitative way). Thus, in particular also $y^*_k\to T^\circ x^*$ for $k\to\infty$. This yields
\[
\norm{x^*-J^S_{\mu_0}(x^*+\mu_0y^*_k)}\to\norm{x^*-J^S_{\mu_0}(x^*+\mu_0 T^\circ x^*)}
\]
for $k\to\infty$ while item (3) yields $\norm{x^*-J^S_{\mu_0}(x^*+\mu_0y^*_k)}\to 0$. Thus
\[
\norm{x^*-J^S_{\mu_0}(x^*+\mu_0 T^\circ x^*)}=0
\]
and therefore $x^*=J^S_{\mu_0}(x^*+\mu_0 T^\circ x^*)$, i.e. 
\[
T^\circ x^*=\mu_0^{-1}(x^*+\mu_0 T^\circ x^* - x^*)\in Sx^*.
\]
Now, as also $T^\circ x^*\in Tx^*$, we have $Sx^*\cap Tx^*\neq\emptyset$, i.e. $x^*\in\Gamma$.
\end{proof}

The next lemma gives a preliminary result for the extraction of a modulus of uniform quasi-Fej\'er monotonicity, obtained by an extraction from the proof of quasi-Fej\'er monotonicity given by Moudafi in \cite{Mou2015}.
\begin{lemma}\label{lem:qfejerinequality}
Let $n,r\in\mathbb{N}$ and $l\in\mathbb{N}^*$ be given and let $x^*\in X_0$ and $y^*$ be such that
\[
\norm{x^*-J^S_{\mu_{n+k}}(x^*+\mu_{n+k}y^*)}\leq\frac{1}{r+1}
\]
for all $k\in [0;l-1]$. Then
\begin{align*}
\norm{x_{n+l}-x^*}&\leq\prod_{k=0}^{l-1}\left(1+\frac{\mu_{n+k}}{\lambda_{n+k}}\right)\norm{x_n-x^*}\\
&\qquad+(\norm{T^\circ x^*}+\norm{y^*})\sum_{k=0}^{l-1}\mu_{n+k}\prod_{j=k+1}^{l-1}\left(1+\frac{\mu_{n+j}}{\lambda_{n+j}}\right)\\
&\qquad+\sum_{k=1}^l\frac{1}{r+1}\prod_{j=k}^{l-1}\left(1+\frac{\mu_{n+j}}{\lambda_{n+j}}\right).
\end{align*}
\end{lemma}
\begin{proof}
Given $n$ and $r$ as well as $x^*$ and $y^*$, we get
\begin{align*}
\norm{x_{n+1}-x^*}&=\norm{J^S_{\mu_n}(x_n+\mu_nT_{\lambda_n}x_n)-x^*}\\
&\leq\norm{J^S_{\mu_n}(x_n+\mu_nT_{\lambda_n}x_n)-J^S_{\mu_n}(x^*+\mu_ny^*)}+\norm{J^S_{\mu_n}(x^*+\mu_ny^*)-x^*}\\
&\leq\norm{x_n+\mu_nT_{\lambda_n}x_n-x^*-\mu_ny^*}+\frac{1}{r+1}\\
&\leq\norm{x_n-x^*}+\mu_n\norm{T_{\lambda_n}x_n-T_{\lambda_n}x^*}+\mu_n\norm{T_{\lambda_n}x^*-y^*}+\frac{1}{r+1}\\
&\leq\norm{x_n-x^*}+\frac{\mu_n}{\lambda_n}\norm{x_n-x^*}+\mu_n(\norm{T^\circ x^*}+\norm{y^*})+\frac{1}{r+1}\\
&=\left(1+\frac{\mu_n}{\lambda_n}\right)\norm{x_n-x^*}+\mu_n(\norm{T^\circ x^*}+\norm{y^*})+\frac{1}{r+1}.
\end{align*}
where we have used $\lambda_n^{-1}$-Lipschitz continuity of $T_{\lambda_n}$ (Corollary 23.11 in \cite{BC2017}). This generalizes to the claim by induction on $l$.
\end{proof}
This immediately gives a modulus of uniform quasi-Fej\'er monotonicity when we assume certain bounds on the objects involved. For this, we define a \emph{bounded subtraction} $\remin$ by $n\remin m:=\max\{0,n-m\}$.

\begin{lemma}\label{lem:quasiFejerModulus}
Let $M\geq \norm{T^\circ x^*}$ for any $x^*\in X_0$. Further, let $A\geq\sum_{n=0}^\infty\frac{\mu_n}{\lambda_n}$ and assume that $\lambda_n\leq B$ for all $n$. Then $\sum_{n=0}^\infty\mu_n<\infty$ and $(x_n)$ is uniformly quasi-$(e^A\mathrm{id}_{\mathbb{R}_+},\mathrm{id}_{\mathbb{R}_+})$-Fej\'er monotone w.r.t. $\Gamma_k$ with modulus $\chi$, that is for all $r,n,m\in\mathbb{N}$:
\[
\forall x^*\in\Gamma_k\forall l\leq m\left(\norm{x_{n+l}-x^*}<e^A\norm{x_n-x^*}+(2M+1)e^A\sum_{i=n}^{n+l-1}\mu_i+\frac{1}{r+1}\right)
\]
where
\[
k=\chi(r,n,m):=\max\{n+m\remin 1,\ceil{(r+1)\cdot m\cdot e^A}\}.
\]
\end{lemma}
\begin{proof}
Assume $m\geq 1$ without loss of generality. First, $\lambda_n\leq B$ together with
\[
\sum_{n=0}^\infty\frac{\mu_n}{\lambda_n}\leq A<\infty
\]
gives $\sum_{n=0}^\infty\mu_n<\infty$. Let $x^*\in\Gamma_k$ be arbitrary and $y^*\in\Gamma_k(x^*)$. Then
\[
\norm{x^*-J_{\mu_i}^S(x^*+\mu_iy^*)}\leq\frac{1}{k+1}\leq\frac{1}{\ceil{(r+1)\cdot m\cdot e^A}+1}
\]
for all $i\leq n+m-1$, since $k\geq n+m\remin 1$ and $m\geq 1$, and thus in particular for all $i\in[n;n+l-1]$ and any $l\leq m$. By the previous Lemma \ref{lem:qfejerinequality}, using that $\vert\norm{y^*}-\norm{T^\circ x^*}\vert\leq 1$, i.e. $\norm{y^*}\leq\norm{T^\circ x^*}+1\leq M+1$, we get 
\begin{align*}
\norm{x_{n+l}-x^*}
&\leq e^A\norm{x_n-x^*}+(2M+1)e^A\sum_{i=n}^{n+l-1}\mu_i+e^A\sum_{k=1}^l\frac{1}{\ceil{(r+1)\cdot m\cdot e^A}+1}\\
&< e^A\norm{x_n-x^*}+(2M+1)e^A\sum_{i=n}^{n+l-1}\mu_i+e^A\frac{l}{(r+1)\cdot m\cdot e^A}\\
&\leq e^A\norm{x_n-x^*}+(2M+1)e^A\sum_{i=n}^{n+l-1}\mu_i+\frac{1}{r+1}.
\end{align*}
\end{proof}
We move on to the lim-inf-property. For that, we first show a general inequality in the spirit of the proximal point algorithm (see \cite{BC2017}) which requires the following result:
\begin{lemma}[\cite{BC2017}, Proposition 23.31, (i)]\label{lem:fundEqResolvents}
Let $A$ be maximally monotone, $\gamma,\lambda>0$ and $x$ a point. Then
\[
J^A_\gamma x=J^A_{\lambda\gamma}(\lambda x+ (1-\lambda)J^A_\gamma x).
\]
\end{lemma}
\begin{lemma}\label{lem:approxerrorbound}
For any $n,i\in\mathbb{N}$, we have
\[
\norm{x_n-J^S_{\mu_i}(x_n+\mu_i T_{\lambda_n}x_n)}\leq\norm{x_n-x_{n+1}}+\vert\mu_n-\mu_i\vert\frac{\norm{x_n-x_{n+1}}}{\mu_n}.
\]
\end{lemma}
\begin{proof}
First, we have
\begin{align*}
&\norm{x_n-J^S_{\mu_i}(x_n+\mu_iT_{\lambda_n}x_n)}\\
&\qquad\qquad\qquad\leq\norm{x_n-x_{n+1}}+\norm{J^S_{\mu_n}(x_n+\mu_nT_{\lambda_n}x_n)-J^S_{\mu_i}(x_n+\mu_iT_{\lambda_n}x_n)}\\
&\qquad\qquad\qquad\leq\norm{x_n-x_{n+1}}+\norm{J^S_{\mu_n}x^n_n-J^S_{\mu_i}x	^i_n}
\end{align*}
where we write $x_n^j = x_n + \mu_jT_{\lambda_n}x_n$ for any $j$ for simplicity. Then, we have
\begin{align*}
&\norm{ J^S_{\mu_n}x_n^n-J^S_{\mu_i}x_n^i}\\
&\quad=\norm{J^S_{\mu_i}\left(\frac{\mu_i}{\mu_n}x_n^n+\left(1-\frac{\mu_i}{\mu_n}\right)J^S_{\mu_n}x_n^n\right)-J^S_{\mu_i}x_n^i} \quad\text{(by Lemma \ref{lem:fundEqResolvents})}\\
&\quad\leq\norm{\frac{\mu_i}{\mu_n}x_n^n+\left(1-\frac{\mu_i}{\mu_n}\right)J^S_{\mu_n}x_n^n-x_n^i} \quad\text{(nonexpansivity of }J^S_{\mu_i})\\
&\quad=\norm{\frac{\mu_i}{\mu_n}(x_n+\mu_nT_{\lambda_n}x_n)+\left(1-\frac{\mu_i}{\mu_n}\right)x_{n+1}-(x_n+\mu_iT_{\lambda_n}x_n)}\\
&\quad=\norm{\frac{\mu_i}{\mu_n}x_n-x_n+\left(1-\frac{\mu_i}{\mu_n}\right)x_{n+1}}\\
&\quad=\vert\mu_n-\mu_i\vert\frac{\norm{x_n-x_{n+1}}}{\mu_n}.
\end{align*}
Combined, we have
\begin{align*}
\norm{x_n-J^S_{\mu_i}(x_n+\mu_iT_{\lambda_n}x_n) }\leq\norm{x_n-x_{n+1}}+\vert\mu_n-\mu_i\vert\frac{\norm{x_n-x_{n+1}}}{\mu_n}.
\end{align*}
\end{proof}
Under suitable quantitative reformulations of the assumptions, we get the following result translating a quantitative version of the assumption (3) of Theorem \ref{thm:MoudafiConv} into a $\liminf$-bound.
\begin{lemma}\label{lem:liminfbound}
Let $C\geq 1$ be an upper bound on both $\mathrm{diam}(\mu_n)$ and $(\mu_n)$ and let $M\in\mathbb{N}^*$ be such that $M\geq\norm{T^\circ x^*}$ for any $x^*\in X_0$. Further, let $\phi$ be s.t.
\[
\forall k,n\exists N\in [n;\phi(k,n)]\left(\norm{x_N-x_{N+1}}/\mu_N<\frac{1}{k+1}\right)
\]
and such that it is monotone w.r.t. $k$ and $n$. Let $\theta$ be a rate of convergence for $\lambda_n\to 0$, i.e.
\[
\forall k\forall n\geq\theta(k)\left(\lambda_n\leq\frac{1}{k+1}\right).
\]
Further, let $\varpi$ be a modulus of uniform continuity for $T$ w.r.t. $H^*$. Then the function
\[
\Phi(k,n)=\phi\left(\left\lceil 2C(k+1)\right\rceil-1, \max\{\theta(M\varpi(k)+M-1),n\}\right)
\]
is a $\liminf$-bound for $x_n$ w.r.t. $\Gamma_k$.
\end{lemma}
\begin{proof}
Clearly, by assumption on $\phi$, there exists an $N\in [\max\{\theta(M\varpi(k)+M-1),n\};\Phi(k,n)]$ with
\[
\frac{\norm{x_N-x_{N+1}}}{\mu_N}<\frac{1}{\left\lceil 2C(k+1)\right\rceil- 1+1}\leq\frac{1}{2C(k+1)}.
\]
Thus, we have
\[
\norm{x_N-x_{N+1}}=\mu_N\frac{\norm{x_N-x_{N+1}}}{\mu_N}<C\frac{1}{2C(k+1)}=\frac{1}{2(k+1)}.
\]
Thus, we get
\begin{align*}
\norm{x_N-J^S_{\mu_i}(x_N+\mu_i T_{\lambda_N}x_N)}&\leq\norm{x_N-x_{N+1}}+\vert\mu_N-\mu_i\vert\frac{\norm{x_N-x_{N+1}}}{\mu_N}\\
&<\frac{1}{2(k+1)}+C\frac{1}{2C(k+1)}\\
&=\frac{1}{k+1}
\end{align*}
for all $i\leq k$. Now, we have $H^*[TJ^T_{\lambda_N}x_N,Tx_N,1/(k+1)]$: using accretivity of the operator $T$, it is easy to see that
\[
\norm{x_N-J^T_{\lambda_N}x_N}\leq\lambda_N\norm{T^\circ x_N}\leq\lambda_NM
\]
and as $N\geq\theta(M\varpi(k)+M-1)$, we have 
\[
\lambda_N\leq\frac{1}{M(\varpi(k)+1)}
\]
and thus $\norm{x_N-J^T_{\lambda_N}x_N}\leq\frac{1}{\varpi(k)+1}$. By assumption on $\varpi$, we have $H^*[TJ^T_{\lambda_N}x_N,Tx_N,1/(k+1)]$. This implies $H^*[T_{\lambda_N}x_N,Tx_N,1/(k+1)]$ since we have $H^*[T_{\lambda_N}x_N,TJ^T_{\lambda_N}x_N,0]$ as $T_{\lambda_N}x_N\in TJ^T_{\lambda_N}x_N$. Therefore, in particular we have $\norm{T_{\lambda_N}x_N-z}\leq\frac{1}{k+1}$ for some $z\in Tx_N$ and therefore
\begin{align*}
\frac{1}{k+1}&\geq\norm{T_{\lambda_N}x_N-z}\\
&\geq\vert\norm{T_{\lambda_N}x_N}-\norm{z}\vert\\
&=\norm{z}-\norm{T_{\lambda_N}x_N}\\
&\geq\norm{T^\circ x_N}-\norm{T_{\lambda_N}x_N}\\
&=\vert\norm{T^\circ x_N}-\norm{T_{\lambda_N}x_N}\vert
\end{align*}
since $z\in Tx_N$ and thus $\norm{T_{\lambda_N}x_N}\leq\norm{T^\circ x_N}\leq\norm{z}$. Thus, we have shown $x_N\in\Gamma_k$.
\end{proof}
\begin{theorem}\label{thm:metastabbasecase}
Let $T,S$ be two maximally monotone operators on a finite dimensional Hilbert space $X$ such that $\mathrm{dom}S\subseteq\mathrm{dom}T$. Let $M\in\mathbb{N}^*$ be such that $M\geq \norm{T^\circ x^*}$ for any $x^*\in X_0$ and let $\varpi$ be a modulus of uniform continuity for $T$ w.r.t. $H^*$. Further, let $A\geq\sum_{n=0}^\infty\frac{\mu_n}{\lambda_n}$ and assume that $\lim_{n\to \infty}\lambda_n= 0$ with a rate of convergence $\theta$. Let $C\geq 1$ be an upper bound on both $\mathrm{diam}(\mu_n)$ and $(\mu_n)$. Further, let $\phi$ be s.t.
\[
\forall k,n\exists N\in [n;\phi(k,n)]\norm{x_N-x_{N+1}}/\mu_N<\frac{1}{k+1}
\]
and such that it is monotone w.r.t. $k$ and $n$. Let $L\geq\mathrm{diam}(x_n)$ and let $\xi$ be a Cauchy rate for $\sum_n\mu_n<\infty$. Then $(x_n)$ is Cauchy and, moreover, for any $k\in\mathbb{N}$ and any $g:\mathbb{N}\to\mathbb{N}$:
\[
\exists N\leq\Psi(k,g,\Phi,\chi,\tilde\xi,A,L)\forall i,j\in [N;N+g(N)]\left(\norm{x_i-x_j}\leq\frac{1}{k+1}\right)
\]
where $\Psi(k,g,\Phi,\chi,\tilde\xi,A,L)=\Psi_0(P,k,g,\Phi,\chi,\tilde\xi)$ defined by recursion with
\[
\begin{cases}
\Psi_0(0,k,g,\Phi,\chi,\tilde\xi)=0\\
\Psi_0(n+1,k,g,\Phi,\chi,\tilde\xi)=\Phi(\chi^M_g(\Psi_0(n,k,g,\Phi,\chi,\tilde\xi),8k+7,\tilde\xi(8k+7))
\end{cases}
\]
with $P=\ceil{2\ceil{8e^A(k+1)}\sqrt{d}L}^d+1$ where $d$ is the dimension of $X$, $\tilde\xi(n)=\xi(\ceil{(2M+1)e^A(n+1)}-1)$ and
\[
\Phi(k,n)=\phi\left(\left\lceil 2C(k+1)\right\rceil-1, \max\{\theta(M\varpi(k)+M-1),n\}\right)
\]
as well as
\begin{gather*}
\chi(r,n,m)=\max\{n+m\remin 1,\ceil{(r+1)\cdot m\cdot e^A}\},\\
\chi_g(n,k)=\chi(n,g(n),k),\chi_g^M(n,k)=\max\{\chi_g(i,k)\mid i\leq n\}.
\end{gather*}
Further, for any $k$ and any $g$ as above:
\[
\exists N\leq\Psi'(k,g,\Phi,\chi,\widetilde{\xi},A,L)\forall i,j\in [N;N+g(N)]\left(\norm{x_i-x_j}\leq\frac{1}{k+1}\text{ and }x_i\in\Gamma_k\right),
\]
where $\Psi'(k,g,\Phi,\chi,\widetilde{\xi},A,L)=\Psi(k_0,g,\Phi,\chi_{k},\widetilde{\xi},A,L)$ with $\Psi$ as before and with
\[
k_0:=\max\left\{k,\ceil*{\frac{\omega(k)-1}{2}}\right\}
\]
where
\[
\omega(k):=\max\{\varpi(2k+1),4k+3,\varpi(4M((k+1)^2)-1))\}
\]
and with
\[
\chi_k(r,n,m):=\max\{\delta(k),\chi(r,n,m)\}\text{ where }\delta(k):=2k+1
\]
with $\chi$ as before.
\end{theorem}
\begin{proof}
The theorem arises from a direct application of Theorem 6.4 from \cite{KLN2018} with $X:=X_0$, $F:=\Gamma\cap X_0$ and $AF_k:=\Gamma_k$ with $G:=e^A\mathrm{id}_{\mathbb{R}_+}$ and $H:=\mathrm{id}_{\mathbb{R}_+}$. Note for this that if $\xi$ is a Cauchy modulus for $\sum_n\mu_n<\infty$, then $n\mapsto\xi(\ceil{(2M+1)e^A(n+1)}-1)$ is a Cauchy modulus for
\[
(2M+1)e^A\sum_n\mu_n<\infty
\]
as we have
\[
\sum^\infty_{i=\xi(\ceil{(2M+1)e^A(n+1)}-1)}(2M+1)e^A\mu_i<(2M+1)e^A\frac{1}{\ceil{(2M+1)e^A(n+1)}}\leq\frac{1}{n+1}.
\]
Further, by Example 2.8 in \cite{KLN2018}, $P$ is correctly defined since $\left\lceil2(k+1)\sqrt{d}L\right\rceil^d$ is a modulus of total boundedness of $\overline{B}(0;L)$ and therefore also of $\overline{B}(x_0;L)=\overline{B}(0;L)+x_0$ as moduli of total boundedness are easily seen to be translation invariant over normed spaces. This clearly makes it a modulus of total boundedness for $X_0$ as well. Lemma \ref{lem:liminfbound} gives that $\Phi$ is a $\liminf$-bound and Lemma \ref{lem:quasiFejerModulus} gives that $\chi$ is a modulus of uniform quasi-Fej\'er monotonicity.\\

The second claim can be concluded from the first claim in the same way that Theorem 5.3 in \cite{KLN2018} is proved: $\chi_k$ is still a modulus of uniform quasi-Fej\'er monotonicity and thus
\[
\exists N\leq\Psi'\forall i,j\in [N;N+g(N)]\left( d(x_i,x_j)\leq\frac{1}{k_0+1}\leq\frac{1}{k+1}\right)
\]
by the first result where actually, by inspecting the proof of Theorem 6.4 given in \cite{KLN2018}, there is an index $n$ such that
\begin{enumerate}
\item $x_n\in\Gamma_{(\chi_k)_g(N,m)}$ for some $m$,
\item $\forall i\in [N;N+g(N)]\left( d(x_i,x_n)\leq\frac{1}{2k_0+2}\leq\frac{1}{\omega(k)+1}\right)$.
\end{enumerate}
As $(\chi_k)_g(N,m)=\chi_k(N,g(N),m)\geq\delta(k)$, we get $x_n\in\Gamma_{\delta(k)}$. By Lemma \ref{lem:intersectionGamma}, as $\delta$ and $\omega$ are moduli of uniform closedness, we get $x_i\in\Gamma_k$ for all $i\in [N;N+g(N)]$.
\end{proof}
The above theorem is a finitization of Theorem \ref{thm:MoudafiConv} under the additional assumption that $T$ is uniformly continuous w.r.t $H^*$ (which is suggested by the logical metatheorems used to obtain this analysis, see Remark \ref{rem:logic} for a further discussion of this): Assume that we are in the situation of the conclusion of the above theorem. The metastability of $(x_n)$ trivially (but non-effectively) implies that $(x_n)$ is Cauchy and thus convergent to some $x$.

By Lemma \ref{lem:intersectionGamma}, we obtain that $\bigcap_k\Gamma_k$ is uniformly closed w.r.t. $\Gamma_k$. Further, Lemma \ref{lem:liminfbound} implies that the above sequence has the $\liminf$-property with respect to the sequence $(\Gamma_k)$.

By Lemma 3.6 of \cite{KLN2018}, we get that $x\in\bigcap_k\Gamma_k\subseteq\Gamma$ and $x$ is therefore a solution to the original problem.\\

There are two further interesting notes to make here:
\begin{enumerate}
\item The assumption that $\lim_{n\to\infty}\norm{x_n-x_{n+1}}/\mu_n=0$ was weakened to \[
\liminf_{n\to\infty}\norm{x_n-x_{n+1}}/\mu_n=0.
\]
\item The assumption $\Gamma\neq\emptyset$ got weakened to the existence of a bound $L\geq\mathrm{diam}(x_n)$. This is indeed a weakening as if $\Gamma\neq\emptyset$, let $p\in\Gamma$ and $q\in Tp\cap Sp$. Then we get $\norm{x_n}\leq\norm{x_n-p}+\norm{p}$ and the former term $\norm{x_n-p}$ can be bounded using quasi-Fej\'er monotonicity as established in Lemma \ref{lem:qfejerinequality} by
\[
\norm{x_n-p}\leq e^A\norm{x_0-p}+(\norm{T^\circ p}+\norm{q})e^A\sum_{k}\mu_k<\infty
\]
with $A$ as in the above theorem and thus $\norm{x_n}$ is even bounded in that case.
\end{enumerate}
\begin{remark}[For logicians]\label{rem:logic}
The analysis of Moudafi's result as presented above can be explained by, and was obtained using, the general logical metatheorems for the extraction of uniform bounds from noneffective proofs involving set-valued operators developed in \cite{Pis2022a}. Further, the main analytical tools used in the proof of Theorem \ref{thm:MoudafiConv} given in \cite{Mou2015} have interesting connections to new proof theoretic notions introduced in \cite{Pis2022a} (where this case study was instrumental in uncovering these connections). We expect that these connections will influence future approaches to quantitative results in monotone operator theory and we thus want to detail them in the following, motivated by a discussion of the logical aspects of this case study of Moudafi's algorithm. In particular, we want to focus on
\begin{enumerate}
\item the operator $T^\circ(x)=P_{Tx}(0)$,
\item the use of the closure of the graph of both $T$ and $S$,
\item the assumption that $T$ is bounded on bounded sets.
\end{enumerate}

At first, the analysis presented in the previous parts of the paper can formalized in (extensions of) the system $\mathcal{T}^\omega$ introduced in \cite{Pis2022a}. These extensions amount to the treatment of two monotone operators together with the treatment of the uniform continuity of $T$ w.r.t. $H^*$ and the operator $T^\circ$ (see the discussion in \cite{Pis2022a}).

Now, the operator $T^\circ$ can be treated in the context of the logical metatheorems by a suitable additional constant of type $X(X)$ where $X$ is an additional abstract type for the respective Hilbert space (see \cite{GeK2008,Koh2005}) together with characterizing axioms (see \cite{Pis2022a})
\begin{enumerate}[(i)]
\item $\forall x^X(x\in\mathrm{dom}T\rightarrow T^\circ x\in Tx)$,
\item $\forall x^X,y^X(y\in Tx\rightarrow \langle y-T^\circ x,-T^\circ x\rangle\leq 0)$.
\end{enumerate}
For that, the functional $T^\circ$ (extended to the whole space by $T^\circ x:=0$ for $x\not\in\mathrm{dom}T$) needs to be majorizable in the sense of \cite{GeK2008,Koh2005} (see also \cite{Koh2008} for various perspectives on this), i.e. there needs to exist a function $f:\mathbb{N}\to\mathbb{N}$ such that 
\[
f\text{ is nondecreasing and }\forall x\in X,n\in\mathbb{N}\left(\norm{x}\leq n\rightarrow\norm{T^\circ x}\leq fn\right). 
\]
This turns out to connect intimately with a notion of majorizability for the set-valued operator $T$ introduced in \cite{Pis2022a}: Call a set-valued operator $T$ \emph{majorizable} if there exists a selection function $t:X\to X$ such that
\[
tx\in Tx\text{ for any }x\in\mathrm{dom}T
\]
and such that $t$ is majorizable in the sense of the above. Then $T$ is majorizable if and only if $T^\circ$ is majorizable.

Even further, the notion of $T$ being bounded on bounded sets can be recognized as a uniform majorizability assumption (see also \cite{Pis2022a}): $T$ is bounded on bounded sets if and only if it is uniformly majorizable in the sense that there exists some $f:\mathbb{N}\to\mathbb{N}$ such that any $t:X\to X$ with $tx\in Tx$ for $x\in\mathrm{dom}T$ and $tx=0$ otherwise is majorized by $f$.

So, majorizability of $T$ (or, equivalently, $T^\circ$) is already guaranteed by the much stronger assumption of uniform majorizability in Theorem \ref{thm:MoudafiConv}. However, as a consequence of the proof-theoretic analysis, this assumption of $T$ being bounded on bounded sets can be weakened to plain majorizability of $T$. Note that this was essentially also observed in \cite{RAC2020} from an analytic perspective in the context of extensions of Moudafi's result.

In the above analysis, this majorant is represented by the bound $M$: The bound is defined via the property $M\geq\norm{T^\circ x^*}$ for any $x^*\in X_0\subseteq\overline{B}_L(x_0)$. Given a majorant $f$ of $T$/$T^\circ$, this is (a bound on) the value of $f$ on (a bound on) $L+\norm{x_0}$. So, we see that in this concrete situation, even only local information on the majorant is required. Note again that this was also observed in \cite{RAC2020} from an analytical perspective.

Further discussions regarding this new notion of (uniform) majorizability of set-valued operators as well as its connections to other notions from monotone operator theory via a proof-theoretic perspective will be given in \cite{KP2022}.\\

Now, item (2), i.e. the closure of the graph of a maximally monotone operator, turns out to be equivalent to the extensionality principle
\[
\forall x,y\in X\left( x=y\rightarrow Tx=Ty\right)
\]
of the set-valued operator $T$ over the system $\mathcal{T}^\omega$ as shown in \cite{Pis2022a}. It is well known that extensionality can not be provable in systems which allow for the extraction of (uniform) bounds from proofs like, e.g., $\mathcal{T}^\omega$ and its extensions (see \cite{Koh2008} for various general discussions of this) and thus a uniform quantitative version of extensionality, namely some uniform continuity principle, has to be added.

As already discussed in \cite{KP2020}, there are certain problems with formulating one of the most widely known version of uniform continuity of a set-valued operator defined via the Hausdorff-metric $H$ (see \cite{MN2001})
\[
\forall\varepsilon>0\exists\delta>0\forall x,y\in\mathrm{dom}T\left(\norm{x-y}\leq\delta\rightarrow H(Tx,Ty)\leq\varepsilon\right).
\]
Motivated by this, the weaker notion of uniform continuity w.r.t. $H^*$ as discussed before is introduced in \cite{KP2020}. This uniform continuity w.r.t. $H^*$ can be added as an axiom to the system $\mathcal{T}^\omega$ such that one still obtains a bound extraction result (see \cite{Pis2022a}). Even further, this was recognized in \cite{Pis2022a} to be the uniform quantitative version of the following weak approximate extensionality principle
\[
\forall x,y\in X\left(x=y\rightarrow \forall k\in\mathbb{N}\left( H^*\left[Tx,Ty,\frac{1}{k+1}\right]\right)\right).
\]
This principle can be used in place of the full extensionality principle in some situations, for example whenever the rest of the proof following the application of extensionality is extensional in the variables.

Now, in Moudafi's proof, the application of extensionality of $T$ in form of the closure of the graph of $T$ can actually be recognized as just an application of this approximate extensionality principle:  extensionality of $T$ is used to conclude $y \in Tx$ given convergent subsequences $T_{\lambda_{n_k}} x_{n_k} \to y$ and $J^T_{\lambda_{n_k}}x_{n_k} \to x$ and using that $T_{\lambda_{n_k}} x_{n_k} \in T(J^T_{\lambda_{n_k}}x_{n_k})$. The rest of the proof is extensional as well as continuous in $y$ and this thus reduces to an application of the above approximate extensionality principle. The metatheorems then immediately upgrade $T$ to being uniformly continuous w.r.t. $H^*$ and a modulus for this crucially features in the analysis presented above.

Further, in Moudafi's proof, extensionality of $S$ can actually be completely avoided by instead using the resolvent and the fact this the resolvent is itself provably extensional (see \cite{Pis2022a}).
\end{remark}
\section{Moduli of regularity and rates of convergence}
\subsection{General theorems on rates of convergence}
As mentioned in the introduction, Fej\'er monotone sequences, in general, do not have a computable rate of convergence. However, the existence of such can be guaranteed in some situations where additional quantitative assumptions are present. Choices for such were extensively studied in \cite{KLAN2019} under the very general notion of \emph{moduli of regularity} (generalizing moduli of uniqueness and other regularity notions known from optimization like error bounds, weak sharp minima and metric subregularity, see the discussion in \cite{KLAN2019}) and based on a proof theoretic perspective, \cite{KLAN2019} presents theorems converting  such moduli of regularity for Fej\'er monotone sequences, modulo some additional minor quantitative assumptions, into rates of convergence for the sequence. To apply these results in our context, we first extend the main quantitative result from \cite{KLAN2019} to the case of quasi-Fej\'er monotone sequences.\\

For that, we follow the setup and notation from \cite{KLAN2019} (which is conflicting with the notation used in the previous section which was derived from \cite{KLN2018}, but the context will make it clear which meaning is intended): let $(X,d)$ be a metric space, $F:X\to\overline{\mathbb{R}}$ with $\overline{\mathbb{R}}=\mathbb{R}\cup\{-\infty,+\infty\}$ be a mapping and assume $\mathrm{zer}F\neq\emptyset$ where $\mathrm{zer}F$ is the set of zeros of $F$.
\begin{definition}[\cite{KLAN2019}]
Let $z\in\mathrm{zer}F$ and $r>0$. A function $\phi:(0,\infty)\to(0,\infty)$ is a \emph{modulus of regularity for $F$ w.r.t. $\mathrm{zer}F$ and $\overline{B}(z;r)$} if for all $\varepsilon>0$ and $x\in\overline{B}(z;r)$:
\[
\vert F(x)\vert <\phi(\varepsilon)\text{ implies }D(x,\mathrm{zer} F)<\varepsilon
\]
where $D$ is the distance function between points and sets. It is a \emph{modulus of regularity for $F$ w.r.t. $\mathrm{zer}F$} if there is a $z$ such that it is a modulus of regularity for $F$ w.r.t. $\mathrm{zer}F$ and $\overline{B}(z;r)$ for any $r>0$.
\end{definition}
Adapting \cite{KLN2018}, we introduce $G$- and $H$-moduli for the functions $G,H$ in the generalized notion of quasi-Fej\'er monotonicity: a function $\alpha_G:\mathbb{R}^*_+\to\mathbb{R}^*_+$ is a \emph{$G$-modulus for $G$} if 
\[
\forall\varepsilon >0\forall a\in\mathbb{R}_+\left( a\leq\alpha_G(\varepsilon)\text{ implies }G(a)\leq\varepsilon\right)
\]
and $\beta_H:\mathbb{R}^*_+\to\mathbb{R}^*_+$ is a \emph{$H$-modulus for $H$} if 
\[
\forall\varepsilon >0\forall a\in\mathbb{R}_+\left( H(a)\leq\beta_H(\varepsilon)\text{ implies }a\leq\varepsilon\right).
\]
For convenience, we assume that a Cauchy rate for a sequence $\sum_i\varepsilon_i<\infty$ is now a mapping taking real values, i.e. a Cauchy rate will now be a function $\xi:\mathbb{R}^*_+\to\mathbb{N}$ which fulfills
\[
\sum_{i=\xi(\delta)}^\infty\varepsilon_i<\delta
\]
for any $\delta\in\mathbb{R}^*_+$.

We then can generalize the main result of \cite{KLAN2019}, i.e. Theorem 4.1, to quasi-$(G,H)$-Fej\'er monotone sequences:
\begin{theorem}\label{thm:modregthm}
Let $(X,d)$ be a metric space and $F:X\to\overline{\mathbb{R}}$ with $\mathrm{zer}F\neq\emptyset$. Let $(x_n)$ be quasi-$(G,H)$-Fej\'er monotone w.r.t. $\mathrm{zer}F$. Let $\alpha_G$ be a G-modulus for $G$, $\beta_H$ be an $H$-modulus for $H$ and let $\beta'_H$ be such that
\[
H(x)\leq a\text{ implies }x\leq\beta'_H(a).
\]
for any $x,a\in\mathbb{R}_+$. Let $b\geq G(d(x_0,z))$ for some $z\in\mathrm{zer}F$, $e\geq\sum_n\varepsilon_n$ and suppose there is a $\tau$ such that
\[
\forall\delta >0\forall n\in\mathbb{N}\exists N\in [n;\tau(\delta,n)](\vert F(x_N)\vert <\delta).
\]
Let $\phi$ be a modulus of regularity for $F$ w.r.t. $\mathrm{zer}F$ and $\overline{B}(z;\beta'_H(b+e))$ and let $\xi$ be a Cauchy rate for $\sum_i\varepsilon_i<\infty$.
Then $(x_n)$ is Cauchy with Cauchy modulus $\theta$:
\[
\forall \delta>0\forall n,m\geq\theta(\delta):=\tau\left(\phi\left(\alpha_G\left(\frac{\beta_H(\delta/2)}{2}\right)\right),\xi\left(\frac{\beta_H(\delta/2)}{2}\right) \right)(d(x_n,x_m)<\delta).
\]
\end{theorem}
\begin{proof}
The proof is a straightforward modification of that of Theorem 4.1 from \cite{KLAN2019}: Let $\delta>0$ be given. By quasi-$(G,H)$-Fej\'er monotonicity, we have 
\[
H(d(x_n,z))\leq G(d(x_0,z))+\sum \varepsilon_i\leq b+e,
\]
i.e. $(x_n)\subseteq\overline{B}(z;\beta'_H(b+e))$. By assumption we have
\[
\exists N\in \left[\xi\left(\frac{\beta_H(\delta/2)}{2}\right);\theta(\delta)\right]\left(\vert F(x_N)\vert <\phi\left(\alpha_G\left(\frac{\beta_H(\delta/2)}{2}\right)\right)\right).
\]
As $\phi$ is a corresponding modulus of regularity, we get
\[
D(x_N,\mathrm{zer} F)<\alpha_G\left(\frac{\beta_H(\delta/2)}{2}\right)
\]
and therefore, there exists a $y\in\mathrm{zer} F$ with $d(x_N,y)<\alpha_G\left(\frac{\beta_H(\delta/2)}{2}\right)$. This yields $G(d(x_N,y))\leq\frac{\beta_H(\delta/2)}{2}$. Thus for any 
\[
n\geq N\geq\xi\left(\frac{\beta_H(\delta/2)}{2}\right),
\]
by quasi-$(G,H)$-monotonicity, we obtain
\begin{align*}
H(d(x_n,y))&\leq G(d(x_N,y))+\sum_{i=N}^\infty\varepsilon_i\\
&\leq \beta_H(\delta/2)/2 + \beta_H(\delta/2)/2\\
&\leq\beta_H(\delta/2)
\end{align*}
and thus $d(x_n,y)\leq\delta/2$ for any such $n$, i.e. in particular $d(x_n,x_m)\leq\delta$ for any $n,m\geq\theta(\delta)$.
\end{proof}
\begin{remark}
$\beta'_H$ is a quantitative version of the property $(H1)$ from \cite{KLN2018} (see Lemma 4.2 there). As apparent from the statement and proof, however, we don't need the full function but only its value at $b+e$. In fact, even an upper bound $B\geq\beta'_H(b+e)$ is sufficient as long as $\phi$ is a modulus of regularity for $F$ w.r.t. $\mathrm{zer}F$ and $\overline{B}(z;B)$.
\end{remark}
\subsection{An application to Moudafi's algorithm}
In applications to problems involving zeros of set-valued operators, \cite{KLAN2019} describes the following approach for phrasing the corresponding problems in terms of the setup introduced above: Define $F_A(x):=D(0,A(x))$ for a set-valued operator $A:X\to \mathcal{P}(X)$. If we have that
\[
D(0,A(x))=0\text{ implies }x\in\mathrm{zer}A\tag{$\dagger$}
\]
for all $x\in X$, then we also get $\mathrm{zer} F_A=\mathrm{zer} A$ which makes $F_A$ a suitable instantiations of the abstract function $F$ from before regarding zeros of $A$. This always applies to maximally monotone operators $A$ if $X$ is a Hilbert space as $A(x)$ is then closed for any $x$. For differences $A=T-S$, even with $T,S$ maximally monotone, $A$ may not be maximally monotone anymore but $(T-S)(x)$ is still closed for any $x$ if $T(x),S(x)$ are and thus $(\dagger)$ holds in that case. Thus, we may choose $F_{T-S}$ if we are interested in zeros of differences of two maximally monotone operators.\\

Now, as exhibited in the above analysis of Moudafi's result, the sequence generated by the algorithm actually converges towards some $x$ such that $T^\circ x\in Sx$ under the above additional quantitative assumptions. 

In fact, it is easy to see that $\bigcap_k\Gamma_k=\{x^*\mid T^\circ x^*\in Sx^*\}$ for the previously used approximations $\Gamma_k$: if $T^\circ x^*\in Sx^*$, then immediately $x^*\in\Gamma_k$ for any $k$ by setting $y^*=T^\circ x^*$. Conversely, if $x^*\in\Gamma_k$ for any $k$, we can obtain $T^\circ x^*\in Sx^*$ as in the proof of Lemma \ref{lem:intersectionGamma}. 

The above $F_{T-S}$ is therefore not really faithful regarding the previous analysis and we in turn consider
\[
F_1(x):=\norm{x-J^S_{\mu_0}(x+\mu_0T^\circ x)}
\]
as $x=J^S_{\mu_0}(x+\mu_0T^\circ x)$ if and only if $T^\circ x\in Sx$. Indeed, this is in particular supported by the following lemma which establishes a relation between the approximations $\Gamma_k$ and the property $F_1(x)\leq\frac{1}{k+1}$.
\begin{lemma}\label{lem:conversionGamma}
Let $M\in\mathbb{N}^*$ be such that $M\geq\norm{T^\circ x^*}$ for all $x\in X_0$ and let $B\in\mathbb{N}^*$ with $B\geq\mu_0$. If $x\in\Gamma_{\kappa(k)}$, then $F_1(x)\leq\frac{1}{k+1}$ where
\[
\kappa(k):=4(M+1)(B(4k+4)-1)^2-1.
\]
\end{lemma}
\begin{proof}
Let $x\in\Gamma_{\kappa(k)}$, i.e. there exists a $y$ s.t. $\vert\norm{y}-\norm{T^\circ x}\vert \leq\frac{1}{\kappa(k)+1}$, $H^*[y,Tx,1/(\kappa(k)+1)]$ and $\norm{x-J^S_{\mu_i}(x+\mu_iy)}\leq\frac{1}{\kappa(k)+1}$ for all $i\leq\kappa(k)$.

Then by item (2), there exists a $w\in Tx$ s.t. $\norm{y-w}\leq\frac{1}{\kappa(k)+1}$, i.e.
\[
\vert\norm{w}-\norm{T^\circ x}\vert\leq\vert\norm{y}-\norm{T^\circ x}\vert+\vert\norm{y}-\norm{w}\vert\leq\frac{2}{\kappa(k)+1}.
\]
Therefore
\begin{align*}
\vert\norm{w}^2-\norm{T^\circ x}^2\vert&\leq \vert\norm{w}+\norm{T^\circ x}\vert\vert\norm{w}-\norm{T^\circ x}\vert\\
&\leq (2M+2)\vert\norm{w}-\norm{T^\circ x}\vert\\
&\leq (2M+2)\frac{2}{\kappa(k)+1}\\
&\leq \frac{1}{(B(4k+4)-1)^2}.
\end{align*}
By the modulus of uniqueness for $T^\circ$, Lemma \ref{lem:Tcircmodunifcont}, (2), we get $\norm{w-T^\circ x}\leq\frac{1}{B(4k+4)}$. Thus, as $\kappa(k)+1\geq B(4k+4)$:
\begin{align*}
\norm{y-T^\circ x}&\leq\norm{y-w}+\norm{w-T^\circ x}\\
&\leq \frac{1}{\kappa(k)+1}+\frac{1}{B(4k+4)}\\
&\leq\frac{1}{B(2k+2)}.
\end{align*}
Now using
\[
\norm{x-J^S_{\mu_0}(x+\mu_0y)}\leq\frac{1}{\kappa(k)+1},
\]
we get (using nonexpansivity of $J^S_{\mu_0}$)
\begin{align*}
\norm{x-J^S_{\mu_0}(x+\mu_0T^\circ x)}&\leq\norm{x-J^S_{\mu_0}(x+\mu_0y)}+\norm{J^S_{\mu_0}(x+\mu_0T^\circ x)-J^S_{\mu_0}(x+\mu_0y)}\\
&\leq\frac{1}{\kappa(k)+1}+\mu_0\norm{T^\circ x-y}\\
&\leq\frac{1}{\kappa(k)+1}+\frac{1}{2k+2}\\
&\leq\frac{1}{k+1}
\end{align*}
as $\kappa(k)\geq 2k+1$.
\end{proof}
Having in mind that $\bigcap_k\Gamma_k=\{x^*\mid T^\circ x^*\in Sx^*\}$ as discussed before, another natural version for $F$ is given by
\[
F_2(x):=D(T^\circ x,Sx).
\]
Note also that $F_2$ relates to $F_{T-S}$ in the following way: for any $x$ (in $X_0$), we have $F_{T-S}(x)\leq F_2(x)$. Indeed, this function $F_2$ can be used if we assume a further modulus of uniform continuity for $S$ w.r.t. $H^*$ as the following lemma shows:
\begin{lemma}\label{lem:conversionGammaPrime}
Let $M\in\mathbb{N}^*$ be such that $M\geq\norm{T^\circ x^*}$ for all $x\in X_0$, let $B\in\mathbb{N}^*$ and $B'\in\mathbb{N}$ with $B\geq\mu_0\geq 2^{-B'}$, and let $\widehat{\varpi}$ be a modulus of uniform continuity for $S$ w.r.t. $H^*$. If $x\in\Gamma_{\widehat{\kappa}(k)}$, then $F_2(x)\leq\frac{1}{k+1}$ where
\[
\widehat{\kappa}(k):=\kappa\left(\max\left\{\widehat{\varpi}(2k+1),2^{B'+1}(k+1)\remin 1\right\}\right)
\]
with $\kappa$ from the previous lemma.
\end{lemma}
\begin{proof}
By the previous lemma, we get
\[
\norm{x-J^S_{\mu_0}(x+\mu_0T^\circ x)}\leq\frac{1}{\max\left\{\widehat{\varpi}(2k+1),2^{B'+1}(k+1)\remin 1\right\}+1}.
\]
Thus we get
\[
H^*\left[SJ^S_{\mu_0}(x+\mu_0T^\circ x),Sx,\frac{1}{2k+2}\right]
\]
Thus, there exist a $z\in Sx$ such that
\[
\norm{\frac{1}{\mu_0}(x+\mu_0T^\circ x-J^S_{\mu_0}(x+\mu_0T^\circ x))-z}\leq\frac{1}{2k+2}.
\]
This entails
\begin{align*}
\norm{z-T^\circ x}&\leq\norm{\frac{1}{\mu_0}(x+\mu_0T^\circ x-J^S_{\mu_0}(x+\mu_0T^\circ x))-z}+\frac{1}{\mu_0}\norm{x-J^S_{\mu_0}(x+\mu_0T^\circ x)}\\
&\leq\frac{1}{2k+2}+\frac{1}{\mu_0}\frac{1}{2^{B'+1}(k+1)}\\
&\leq\frac{1}{2k+2}+\frac{1}{2k+2}\\
&\leq\frac{1}{k+1}.
\end{align*}
Thus, as $z\in Sx$, we have $D(T^\circ x,Sx)\leq\frac{1}{k+1}$.
\end{proof}
As an immediate consequence of Theorem \ref{thm:modregthm} together with the above lemmas, we get the following result for converting the quantitative assumptions into a rate of convergence for the sequence given by Moudafi's algorithm.
\begin{theorem}\label{thm:modregthmdiff}
Let $T,S$ be two maximally monotone operators on a Hilbert space $X$ such that $\mathrm{dom}S\subseteq\mathrm{dom}T$ and $\mathrm{zer}(F)\neq\emptyset$ and let $L\geq\mathrm{diam}(x_n)$. Let $A\geq\sum_{n=0}^\infty\frac{\mu_n}{\lambda_n}$. Let $C\geq 1$ be an upper bound on both $\mathrm{diam}(\mu_n)$ and $(\mu_n)$, $B\in\mathbb{N}^*$ with $B\geq\mu_0$ and let $M\in\mathbb{N}^*$ be such that $M\geq\norm{T^\circ x^*}$ for any $x^*\in X_0$. Further, let $\Phi$ be s.t.
\[
\forall k,n\exists N\in [n;\Phi(k,n)]\norm{x_N-x_{N+1}}/\mu_N<\frac{1}{k+1}
\]
and such that it is monotone w.r.t. $k$ and $n$, $\xi$ be a Cauchy rate (with input $k\in\mathbb{N}$) for $\sum_n\mu_n\leq d<\infty$ and $b\geq \norm{x_0-z}$ for some $z\in\mathrm{zer}(T-S)$. Let $\theta$ be a rate of convergence for $\lambda_n\to 0$, i.e.
\[
\forall k\forall n\geq\theta(k)\left(\lambda_n\leq\frac{1}{k+1}\right).
\]
Further, let $\varpi$ be a modulus of uniform continuity for $T$ w.r.t. $H^*$.

Let $\phi$ be a modulus of regularity for $F_1$ w.r.t. $\mathrm{zer}(F_1)$ and $\overline{B}(z;e^Ab+d)$. Then $(x_n)$ is Cauchy with Cauchy modulus
\[
\forall \varepsilon>0\forall n,m\geq\theta(\varepsilon):=\widehat{\Phi}\left(\kappa\left(\ceil*{\frac{1}{\phi\left(\frac{\varepsilon}{4e^A}\right)}}\right),\tilde{\xi}\left(\frac{\varepsilon}{4}\right)\right)(d(x_n,x_m)<\varepsilon)
\]
where
\[
\tilde\xi(\varepsilon)=\xi\left(\ceil*{(2M+1)e^A\left(\ceil*{\frac{1}{\varepsilon}}+1\right)}-1\right)
\]
and
\[
\widehat{\Phi}(k,n):=\Phi\left(\left\lceil 2C(k+1)\right\rceil-1, \max\{\theta(M\varpi(k)+M-1),n\}\right).
\]
as well as
\[
\kappa(k):=4(M+1)(B(4k+4)-1)^2-1.
\]
Moreover, if we assume $\widehat{\varpi}$ to be a modulus of uniform continuity for $S$ w.r.t. $H^*$ and $B'\in\mathbb{N}$ such that $\mu_0\geq 2^{-B'}$, then the above claim holds for $\phi$ being a modulus of regularity for $F_2$ (or $F_{T-S}$) w.r.t $\mathrm{zer}(F_2)$ (or $\mathrm{zer}(F_{T-S})$) and $\overline{B}(z;e^Ab+d)$ and with $\kappa$ replaced by $\widehat{\kappa}$ defined by
\[
\widehat{\kappa}(k)=\kappa\left(\max\left\{\widehat{\varpi}(2k+1),2^{B'+1}(k+1)\remin 1\right\}\right).
\]
\end{theorem}
\begin{proof}
The proof is a direct application of Theorem \ref{thm:modregthm}. We want to note a few things, however. Again, we work over $X_0$. Then Lemma \ref{lem:quasiFejerModulus} actually established quasi-($e^A\mathrm{id}_{\mathbb{R}_+},\mathrm{id}_{\mathbb{R}_+}$)-Fej\'er monotonicity w.r.t. $\mathrm{zer}(F)$. Naturally, $\delta\mapsto \delta/e^A$ is a G-modulus for $e^A\mathrm{id}_{\mathbb{R}_+}$ and $\delta\mapsto\delta$ is a $H$-modulus for $\mathrm{id}_{\mathbb{R}_+}$ and $\beta'_H$ can just be set to be the identity. Further, the function
\[
(\varepsilon,n)\mapsto\widehat{\Phi}\left(\kappa\left(\ceil*{\frac{1}{\varepsilon}}\right),n\right)
\]
is a $\liminf$-modulus for $F_1$ by Lemma \ref{lem:liminfbound} and Lemma \ref{lem:conversionGamma}.
\end{proof}
Note that it is no longer necessary for $X$ to be finite dimensional in this case. This in particular relies on the fact that the previous moduli do not rely on finite dimensionality either.

\section*{Acknowledgments}
This paper is a revised version of parts of my master thesis \cite{Pis2022} written under the supervision of Prof. Dr. Ulrich Kohlenbach at TU Darmstadt. I want to thank Prof. Kohlenbach. His suggestion to work on Moudafi's algorithm proved to be a very interesting case study in the context of proof mining and I have immensely enjoyed and benefited from our various discussions on the subject matter.

\bibliographystyle{plain}
\bibliography{ref}

\end{document}